\documentclass[10pt]{amsart}

\usepackage{amssymb, amsthm, amsmath, amsfonts, amsxtra, mathrsfs, mathtools}
\usepackage{tikz-cd} 

\usepackage{enumitem}
\setlist[description]{leftmargin=\parindent,labelindent=\parindent}
\usetikzlibrary{matrix,arrows,decorations.pathmorphing}
\usepackage{graphicx}
\usepackage{xcolor}
\usepackage{tikz-cd}
\usepackage{enumitem}
\usepackage{url}
\usepackage{hyperref}
\usepackage{float}
\usepackage{colonequals}
\setlist[description]{leftmargin=\parindent,labelindent=\parindent}
\usetikzlibrary{matrix,arrows,decorations.pathmorphing}

\newtheorem{theorem}{Theorem}[section]
\newtheorem{proposition}[theorem]{Proposition}
\newtheorem{lemma}[theorem]{Lemma}
\newtheorem{corollary}[theorem]{Corollary}

\theoremstyle{definition}

\newtheorem{definition}[theorem]{Definition}

\newtheorem{remark}[theorem]{Remark}
\newtheorem*{unnumberedremark}{Remark}

\newtheorem{fact}[theorem]{Fact}

\newtheorem{innercustomthm}{Theorem}
\newenvironment{customthm}[1]
  {\renewcommand\theinnercustomthm{#1}\innercustomthm}
  {\endinnercustomthm}

\def\Bbar{\overline{\mathcal{B}}}
\def\B{\mathcal{B}}

\def\A{\mathcal{A}}
\def\Mbar{\overline{\mathcal{M}}}
\def\M{\mathcal{M}}

\def\d{\partial}
\DeclareMathOperator{\Admbar}{\overline{Adm}}
\DeclareMathOperator{\rt}{rt}
\def\QQ{\mathbb{Q}}
\def\CC{\mathbb{C}}

\title[Holomorphic forms and non-tautological cycles]{Holomorphic forms and non-tautological cycles on moduli spaces of curves}
\author[Arena]{Veronica Arena}
\author[Canning]{Samir Canning}
\author[Clader]{Emily Clader}
\author[Haburcak]{Richard Haburcak}
\author[Li]{Amy Q. Li}
\author[Mok]{Siao Chi Mok}
\author[Tamborini]{Carolina Tamborini}

\begin{document}
	
	\maketitle

 \begin{abstract}
We prove, for infinitely many values of $g$ and $n$, the existence of non-tautological algebraic cohomology classes on the moduli space $\M_{g,n}$ of smooth, genus-$g$, $n$-pointed curves.  In particular, when $n=0$, our results show that there exist non-tautological algebraic cohomology classes on $\M_g$ for $g=12$ and all $g \geq 16$.  These results generalize the work of Graber--Pandharipande and van Zelm, who proved that the classes of particular loci of bielliptic curves are non-tautological and thereby exhibited the only previously-known non-tautological class on any $\M_g$: the bielliptic cycle on $\M_{12}$.  We extend their work by using the existence of holomorphic forms on certain moduli spaces $\Mbar_{g,n}$ to produce non-tautological classes with nontrivial restriction to the interior, via which we conclude that the classes of many new double-cover loci are non-tautological.
 \end{abstract}
	
\section*{Introduction}
While the full cohomology ring of the moduli space $\Mbar_{g,n}$ of stable curves is generally intractable to study, there is a distinguished subring called the tautological ring that admits an explicit generating set and yet is rich enough to contain most cohomology classes coming from natural algebraic cycles.  Specifically, the tautological rings $RH^*(\Mbar_{g,n}) \subseteq H^*(\Mbar_{g,n})$ are defined simultaneously for all $g$ and $n$ as the smallest system of $\QQ$-subalgebras that contains the fundamental classes $[\Mbar_{g,n}]$ and is closed under pushforward by the forgetful and gluing maps between the moduli spaces.  There is an analogous subring $RH^*(\M_{g,n}) \subseteq H^*(\M_{g,n})$, defined as the image of $RH^*(\Mbar_{g,n})$ under restriction.  A natural question, for both $\Mbar_{g,n}$ and $\M_{g,n}$, is to understand when the equality $RH^* = H^*$ occurs.

On $\Mbar_{0,n}$, all cohomology classes are tautological by work of \cite{Keel}, and on $\Mbar_{1,n}$, all even-degree cohomology classes are tautological by work of Petersen \cite{Petersengenus1}.  It follows that all algebraic cohomology classes on $\M_{0,n}$ and $\M_{1,n}$ are tautological. There are several more cases in the literature where, for small values of $g$ and $n$, all algebraic cohomology on $\M_{g,n}$ is shown to be tautological  \cite{PetersenTommasi,FaberI,FaberII,Izadi,PV,CL789,CLstable,CLPste}; see Figure \ref{bigfig} for a summary.

The first example of a non-tautological algebraic class was produced by Graber and Pandharipande \cite{GP}: the class on $\Mbar_{2,20}$ consisting of curves that admit a double cover of a genus-one curve where the marked points are pairwise switched by the covering involution.  Van Zelm \cite{vZ} significantly extended these results, showing that a similarly-defined bielliptic cycle is non-tautological in an infinite family of cases; in particular, his results show that when $g \geq 2$, there are only finitely many pairs $(g,n)$ for which it is possible that $RH^*(\Mbar_{g,n})=H^*(\Mbar_{g,n})$.  Lian \cite{Lian} generalized these results further, showing that loci of higher-degree covers of higher-genus curves also give rise to non-tautological classes. 

Producing non-tautological algebraic classes on the interior $\M_{g,n}$, on the other hand, has proved significantly harder.  Graber--Pandharipande's original cycle restricts to a non-tautological class on $\M_{2,20}$, but the bielliptic cycle considered by van Zelm was only shown to restrict to a non-tautological class on $\M_{g,2m}$ when $g \geq 2$ and $g+m = 12$.  In particular, prior to the present work, there were only eleven pairs $(g,n)$ for which the existence of a non-tautological algebraic cycle on $\M_{g,n}$ was known.

Our main result vastly extends the previous work on non-tautological classes on $\M_{g,n}$, exhibiting such classes in an infinite family of cases; see Figure~\ref{bigfig} for a visualization.

\begin{figure}
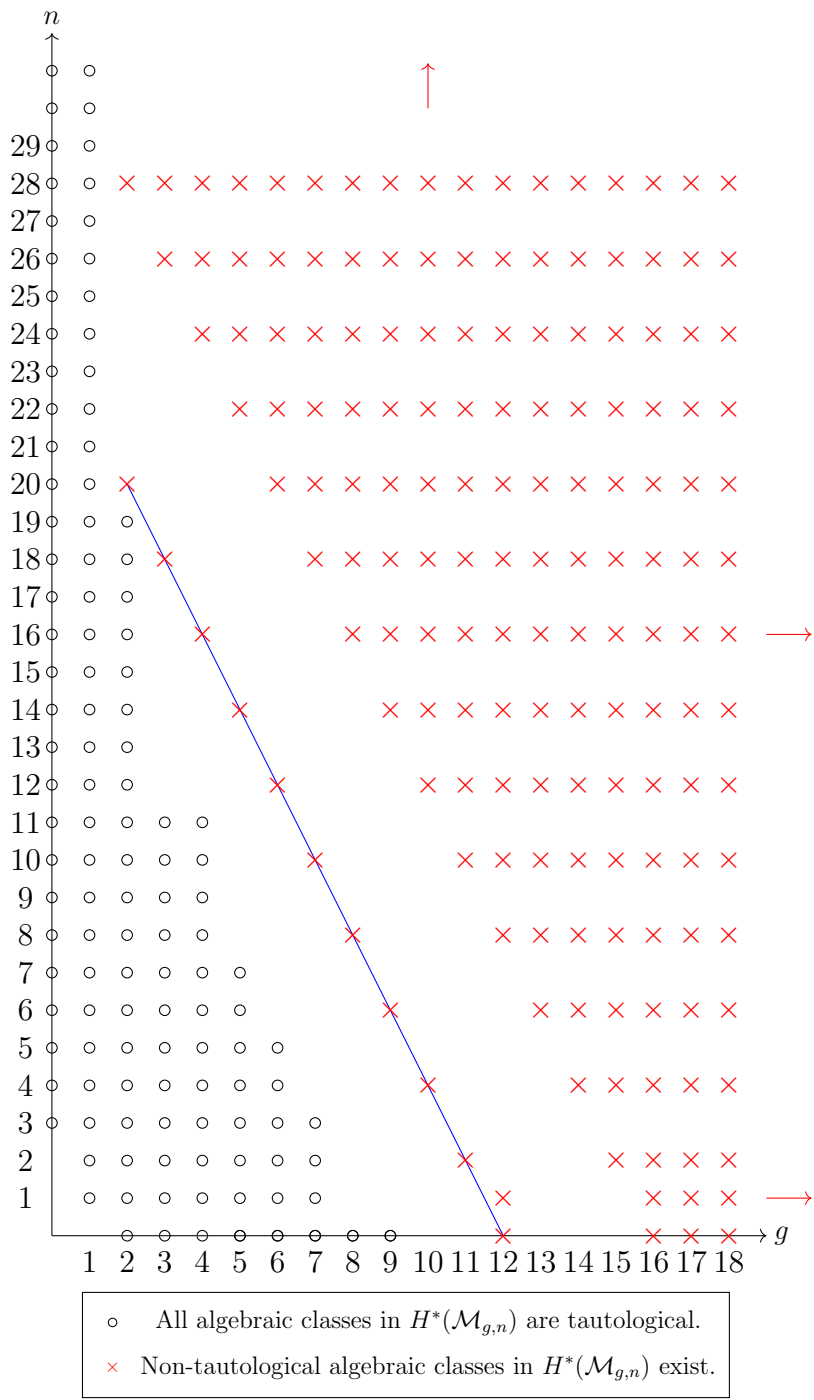
 
\vspace{-.2in}
\begin{center}

\end{center}
 
\caption{The pairs $(g,n)$ for which all algebraic classes on $\M_{g,n}$ are tautological (indicated by circles in the figure) all follow from previous work.  Aside from the eleven pairs on the blue line, all cases where non-tautological algebraic classes on $\M_{g,n}$ exist (indicated by $\color{red}{\times}$'s in the figure) are new to the current work.}
\label{bigfig}
\end{figure}

\begin{customthm}{A}\label{thm:main}
Let $g \geq 4$.  Then there exist non-tautological algebraic cycles in $H^*(\M_{g,2m})$ for any $m \geq 0$ such that either $g+m=12$ or $g+m \geq 16$.  If $g =2$ or $g=3$, the same result holds as long as $g+m$ is even.
\end{customthm}

 \begin{unnumberedremark}
     It was previously known that all algebraic cycles on $\M_g$ are tautological for $g \leq 9$ \cite{Mumford,FaberI, FaberII,Izadi,PV,CL789}; thus, Theorem~\ref{thm:main} settles the question of the existence of non-tautological classes on $\M_g$ for all but the cases $g=10,11,13,14,15$.  The case $g=12$ was the only case in which a non-tautological algebraic cycle on $\M_g$ was previously known \cite{vZ}, namely, the bielliptic cycle $[\B_{12}] \subseteq \M_{12}$.  The bielliptic cycle on $\M_{10}$ and $\M_{11}$, on the other hand, is known to be tautological \cite{CL789,CLbielliptic}.
 \end{unnumberedremark}

\begin{unnumberedremark}
Although Theorem~\ref{thm:main} only shows the existence of non-tautological classes on $\M_{g,n}$ when $n$ is even, it is straightforward to show (see Proposition~\ref{prop:rationaltails}) that, if one is willing to work with the larger moduli space $\M^{\text{rt}}_{g,n}$ of curves with rational tails, then Theorem~\ref{thm:main} can be used to produce such classes with odd numbers of marked points---in particular, on $\M^{\text{rt}}_{g,2m+1}$ for any $(g,m)$ satisfying the hypotheses of Theorem~\ref{thm:main}.  Since $\M^{\text{rt}}_{g,1} = \M_{g,1}$, this shows that there exist non-tautological algebraic classes in $H^*(\M_{g,1})$ whenever $g=12$ or $g \geq 16$, as indicated in Figure~\ref{bigfig}.
\end{unnumberedremark}

In order to prove Theorem~\ref{thm:main}, we begin by following a strategy analogous to that of \cite{GP} and \cite{vZ}: we consider the locus $\B_{g \rightarrow h, n, 2m} \subseteq \M_{g,n+2m}$ of smooth curves admitting a degree-two map to a smooth genus-$h$ curve such that the covering involution fixes the first $n$ marked points and pairwise swaps the last $2m$ marked points.  When $n=0$, there is a particular gluing map
\[i: \Mbar_{h,k} \times \Mbar_{h,k} \rightarrow \Mbar_{g,2m},\]
under which, for certain values of $g, h$, $k$, and $m$, one can prove that
\begin{equation}
    \label{eq:i*vZ}
i^*[\Bbar_{g \rightarrow h, 0, 2m}] = \alpha[\Delta] + B
\end{equation}
for $\Delta$ the diagonal, $\alpha$ a nonzero constant, and $B$ a cycle supported on the boundary.

In the special case when $h=1$ and $k=11$, Graber--Pandharipande and van Zelm used the above strategy to prove that $[\Bbar_{g \rightarrow 1, 0, 2m}]$ is non-tautological.  In that case, the key idea is, first, to use that $H^{11}(\Mbar_{1,11}) \neq 0$ to show that there exists a contribution to the K\"unneth decomposition of $[\Delta]$ that has odd degree and thus is necessarily non-tautological; and second, to use the known structure of $H^*(\Mbar_{1,11})$ to show that a cycle supported on the boundary must be tautological.  Combining these with \eqref{eq:i*vZ} shows that $i^*[\Bbar_{g \rightarrow 1, 0,2m}]$ is non-tautological, and since pullbacks of tautological classes along gluing maps are tautological, this proves that $[\Bbar_{g \rightarrow 1, 0,2m}]$ itself is non-tautological.

For other values of $h$ and $k$, on the other hand, the fact that $B$ is supported on the boundary is no longer sufficient to deduce that it is tautological.  Instead, a more subtle argument is needed: we use the existence, for certain values of $h$ and $k$, of holomorphic forms on $\Mbar_{h,k}$.  These are necessarily non-tautological and not pushed forward from the boundary, so they yield non-tautological contributions to $\alpha[\Delta]$ that cannot be cancelled by $B$.  The result of this reasoning is the following theorem.

\begin{customthm}{B}
    \label{thm:interior}
Let $g\geq 2h$ and set $k := g-2h+m+1$.  Suppose that
\[H^{3h-3+k,0}(\Mbar_{h,k})\neq 0.\]
If $m \geq 1$ or $g \geq 4h+2$, then
     \[[\B_{g\rightarrow h,0,2m}]\in H^{6h-6+2k}(\M_{g,2m})
    \]
    is non-tautological.
\end{customthm}

From here, by applying Theorem~\ref{thm:interior} to the special cases when $h=1$ and $h=2$, we deduce Theorem~\ref{thm:main}.

\begin{unnumberedremark}
Since tautological classes on $\Mbar_{g,2m}$ restrict to tautological classes on $\M_{g,2m}$, it follows from Theorem~\ref{thm:interior} that $[\Bbar_{g \rightarrow h, 0,2m}] \in H^{6h-6+2k}(\Mbar_{g,2m})$ is non-tautological under the same hypotheses.  When $h=1$ and $k \geq 11$ is odd and $k \neq 13$, these hypotheses are satisfied by well-known results on $\Mbar_{1,k}$, as we discuss in Section~\ref{sec:weightfiltration}, and we recover the key base case of \cite{vZ} and \cite{GP}.  When $h=2$ and $k \geq 14$ is even, we prove in Proposition~\ref{prop:g=2weights} that the hypotheses of Theorem~\ref{thm:interior} are also satisfied, yielding a new family of non-tautological classes on the compact moduli space of curves described in Corollary~\ref{cor:g=2}.  This generalizes results of \cite{Lian} to the unpointed case and gives an alternate proof in the pointed case.
\end{unnumberedremark}

\subsection*{Acknowledgements}
We thank Carel Faber, Gerard van der Geer, Siddarth Kannan, Aaron Landesman, Hannah Larson, Carl Lian, Rahul Pandharipande, Dan Petersen, and Johannes Schmitt for helpful conversations, and we thank the anonymous referee for their astute comments. This project started at the AGNES 2023 summer school at Brown University. We thank the organizers for creating the productive research environment that led to this paper and the NSF grant DMS-2312088 for funding the conference. V.A. was supported in part by funds from BSF grant 2018193 and NSF grant DMS-2100548. S.C. was supported by a Hermann-Weyl-Instructorship from the Forschungsinstitut f\"ur Mathematik at ETH Z\"urich.  E.C. was supported by NSF CAREER Grant 2137060. R.H. would like to thank the Hausdorff Research Institute for Mathematics funded by the Deutsche Forschungsgemeinschaft (DFG, German Research Foundation) under Germany's Excellence Strategy (EXC-2047/1 – 390685813) for their generous hospitality during part of the preparation of this work. A.L. was supported in part by NSF grants DMS–2053261 and DMS–2302475. S.C.M. was supported by the Cambridge Commonwealth, European and International Trust. C.T. was partially supported by the Dutch Research Council NWO project ``Cohomology of Moduli Space of Curves" (BM.000230.1), by the INdAM-GNSAGA project CUP
E55F22000270001, and by the DFG-Research Training Group 2553 ``Symmetries and classifying spaces: analytic, arithmetic, and derived."

\section{Preliminaries on double cover cycles}
\label{sec:preliminaries}
In this section, we introduce the double cover cycles of interest and prove the cases in which equation~\eqref{eq:i*vZ} holds.

 \subsection{Double cover cycles}
 \label{subsec addm double cover prelim} 

The loci of double covers we consider are the images of moduli stacks of admissible double covers under forgetful morphisms to the moduli space of curves. We use the definition of the moduli space of admissible covers from \cite{SvZ}, which is a finite cover of the the space of admissible covers defined in \cite{ACV}. 

\begin{definition}
\label{def:admcovers}
    Let $\Admbar(g,h)_{2m}$ be the stack of admissible double covers, which parameterizes tuples 
    \[
    (f:C \rightarrow D; x_1, \ldots, x_r; y_1, y_2, \dots, y_{2m-1}, y_{2m})
    \]
    such that
    \begin{itemize}
        \item $f:C\to D$ is a double cover of connected nodal curves of arithmetic genus $g$ and $h$, respectively;
        \item $x_1, \ldots, x_r \in C$ are precisely the smooth ramification points of $f$;
        \item $y_1,\dots, y_{2m}\in C$ are such that the involution induced by $f$ exchanges $y_{2i-1}$ and $y_{2i}$;
        \item the image under $f$ of each node of $C$ is a node of $D$;
        \item the pointed curves $\big(C; (x_i)_{i=1}^r, (y_i)_{i=1}^{2m}\big)$ and $\big(D; (f(x_i))_{i=1}^r, f(y_{2i-1})_{i=1}^m\big)$ are stable.
    \end{itemize}
\end{definition}

\begin{remark}
\label{rem:r}
    A Riemann--Hurwitz calculation shows that $r = 2g+2-4h$.
\end{remark}

For any $n \leq r$, there is a natural map
\begin{equation}
    \label{eq:phin}
\phi_n:\Admbar(g,h)_{2m}\to \Mbar_{g,n+2m}
\end{equation}
sending an admissible cover to the stabilization of $(C; x_1, \dots, x_n, y_1, \dots, y_{2m})$, and we define
\[\Bbar_{g \rightarrow h, n, 2m}\colonequals\phi_n\left(\Admbar(g,h)_{2m}\right).\]
The map $\Admbar(g,h)_{2m}\to \Mbar_{h,r+m}$ sending an admissible cover to the marked target curve $D$ is finite, so one has
\[\dim\left(\Bbar_{g\to h, n, 2m}\right) = \dim\left(\Mbar_{h,r+m}\right) = 2g+m-h-1.\]
Calculating the codimension thus gives
\[
\left[\Bbar_{g\to h, n, 2m}\right]\in H^{2(g+h+n+m-2)}\left(\Mbar_{g,2m+n}\right).
\]

	\subsection{Double covers and diagonals}
	\label{Sec:Lemma6}

 In this section, we restrict to the case $n=0$, and we set $k\colonequals g-2h+m+1$. Consider the gluing map
	\begin{equation}\label{eq:bargluing}
	i: \Mbar_{h,k}\times \Mbar_{h,k}\rightarrow \Mbar_{g,2m}
	\end{equation}
sending $\big( (C_1; y_1, \ldots, y_k), \; (C_2; z_1, \ldots, z_k) \big)$ to
\[ (C; y_{k-m+1}, z_{k-m+1}, y_{k-m+2}, z_{k-m+2}, \ldots, y_k, z_k),\]
where $C$ is the genus-$g$ curve obtained by pairwise gluing the first $k-m$ marked points of $C_1$ to the first $k-m$ marked points of $C_2$.  We denote by $j$ the composition
 \begin{gather}\label{eq:defj}
   j:\M_{h,k}\times \M_{h,k}\hookrightarrow \Mbar_{h,k}\times \Mbar_{h,k}\xrightarrow{i}\Mbar_{g,2m}.  
 \end{gather}
Generalizing \cite[Lemma 6]{vZ} (which is the case where $h=1$ and $g+m=12$), we prove the following.

\begin{proposition}\label{prop:g=hlemma6}
Let $g\geq 2h$.  If $m \geq 1$ or $g \geq 4h+2$, then 
\[j^*[\Bbar_{g \rightarrow h, 0, 2m}] = \alpha[\Delta]\]
for some $\alpha \in \QQ_{> 0}$, where $\Delta\subseteq \M_{h,k}\times \M_{h,k}$ denotes the diagonal.
	\end{proposition}

As in the proof of \cite[Lemma 6]{vZ}, we begin by reinterpreting the statement via the diagram
	\begin{equation}\label{commutative diagram: gluing maps}
		\begin{tikzcd}
			{\M_{h,k}} \arrow[d, "\delta"'] \arrow[r, "\eta"] & {\Admbar(g,h)_{2m}} \arrow[d, "\phi_0"] \\
			{\M_{h,k}\times \M_{h,k}} \arrow[r, "j"']        & {\Mbar_{g,2m},}                         
		\end{tikzcd}
	\end{equation}
 which is illustrated in Figure~\ref{fig:diagonal_bielliptic_picture}.  Here, $\delta: \M_{h,k} \rightarrow \M_{h,k} \times \M_{h,k}$ is the embedding of the diagonal, $\phi_0$ and $j$ are as in \eqref{eq:phin} and \eqref{eq:defj}, and the map $\eta:\M_{h,k}\rightarrow \Admbar(g,h)_{2m}$ sends a pointed curve $(C_1;x_1,\dots,x_{k})$ to the admissible cover whose source curve $S$ consists of two copies of $C_1$ glued together by rational bridges attached to the first $k-m$ marked points of each copy, and whose target curve $T$ is a single copy of $C_1$ with a rational tail attached to each of the first $k-m$ marked points, where the two copies of $C_1$ in the source curve are swapped by the covering involution.  (The reason for including the rational bridges and rational tails is so that the covering sends nodes in the source curve to nodes in the target curve, as required by Definition~\ref{def:admcovers}.)

\begin{figure}[t]
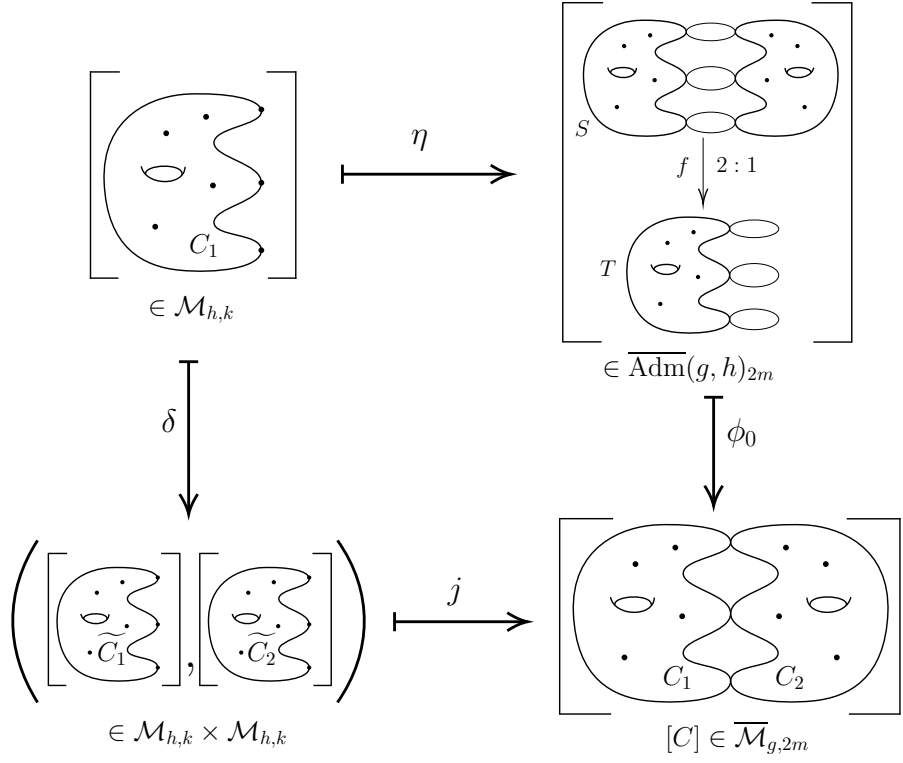

\resizebox{0.75\textwidth}{!}{

}
\caption{An illustration of the commutative diagram \eqref{commutative diagram: gluing maps}.}
\label{fig:diagonal_bielliptic_picture}
\end{figure}

This diagram induces a map $\zeta:\M_{h,k}\rightarrow F$ to the fiber product $F$ of the morphisms $j$ and $\phi_0$,
	\begin{equation*}
		\begin{tikzcd}
                {\M_{h,k}}
                \arrow[bend left]{drr}{\eta}
                \arrow[bend right,swap]{ddr}{\delta}
                \arrow{dr}{\zeta} & & \\
                & F \arrow{r} \arrow{d}{\gamma}
                & {\Admbar(g,h)_{2m}} \arrow{d}{\phi_0} \\
                & {\M_{h,k} \times \M_{h,k}} \arrow[swap]{r}{j}
                & {\Mbar_{g,2m},} 
            \end{tikzcd}
	\end{equation*}
 and via this map, we can reinterpret the conclusion of Proposition~\ref{prop:g=hlemma6} as follows.

\begin{lemma}\label{lem: zeta surjective implies pullback is multiple of diagonal}
If $\zeta:\M_{h,k}\rightarrow F$ is surjective on closed points, then
\[
j^*[\Bbar_{g\rightarrow h,0,2m}]=\alpha[\Delta]
\]
for some $\alpha\in\QQ_{>0}$.
\end{lemma}

\begin{proof}
By definition, we have \[j^*[\Bbar_{g\rightarrow h,0,2m}]=j^*\phi_{0*}[\Admbar(g,h)_{2m}]\in H^{2(g+h+m-2)}(\M_{h,k}\times\M_{h,k}),\] 
where the latter cohomology group is equivalent to
\[H^{2(3h-3+k)}(\M_{h,k}\times\M_{h,k}) = H_{2(3h-3+k)}(\M_{h,k}\times\M_{h,k}).\]
Therefore, since $F$ is a fiber product, the cycle class $j^*[\Bbar_{g\rightarrow h,0,2m}]$ is the pushforward via $\gamma$ of a cycle $Z \in H_{2(3h-3+k)}(F)$. If $\zeta$ is surjective on closed points, then $\dim(F) \leq \dim(\M_{h,k}) = 3h - 3 + k$, so a cycle $Z$ of homological degree $2(3h-3+k)$ can only exist if $\text{dim}(F) = 3h-3+k$ and $Z = c[F]$ for some $c \in \QQ_{>0}$.  This, in turn, means that $Z = \zeta_*(\alpha[\M_{h,k}])$ for some $\alpha \in \QQ_{>0}$, so
\[j^*[\Bbar_{g\rightarrow h,0,2m}] = \gamma_*(Z) = \gamma_*\zeta_*(\alpha[\M_{h,k}]) = \delta_*(\alpha[\M_{h,k}])=\alpha[\Delta]\]
for some $\alpha\in \QQ_{>0}$, as claimed.
\end{proof}

In order to prove that $\zeta$ is surjective on closed points under the hypotheses of Proposition~\ref{prop:g=hlemma6}, we first set some notation, summarized in Figure~\ref{fig:diagonal_bielliptic_picture}. A closed point $P$ of the fiber product $F$ is the data of a curve $\tilde{C}\colonequals (\tilde{C}_1,\tilde{C}_2)$ in $\M_{h,k}\times \M_{h,k}$, an admissible cover $(f:S\rightarrow T)\in \Admbar(g,h)_{2m}$, and an isomorphism
\[
\varphi: j(\tilde{C})\xrightarrow{\sim} \phi_0(S\rightarrow T).
\]
Let $\tilde{j}:\tilde{C_1}\coprod \tilde{C_2}\rightarrow j(\tilde{C})$ be the map induced by $j$. Set $C\colonequals j(\tilde{C})$, $C_1=\tilde{j}(\tilde{C_1})$, and $C_2=\tilde{j}(\tilde{C_2})$.  The involution on $S\rightarrow T$ induces an involution on the stabilization of $S$, and hence on $C$, which we denote by $\tau$. Let $Q_i$ be the node of $C$ corresponding under $j$ to the $i$th marked points of $\tilde{C}_1$ and $\tilde{C}_2$. 

Since $C_1$ and $C_2$ are smooth, either $\tau(C_1) = C_2$ or $\tau$ maps each component to itself. We first consider the case where $\tau(C_1) = C_2$, which is essentially identical to the corresponding argument in \cite{vZ}.

\begin{lemma}\label{lem: tau switching components implies zeta surjective} Let $P  = (\tilde{C}, f: S \rightarrow T, \varphi)$ be a closed point of $F$.  If $\tau(C_1)=C_2$, then $P$ is in the image of $\zeta$.
\end{lemma}

\begin{proof}
In order to show that $P$ is in the image of $\zeta$, we must show that $\tau$ induces an isomorphism between $\tilde{C_1}$ and $\tilde{C_2}$ as pointed curves.  The fact that $\tau(C_1) = C_2$ already implies that $\tilde{C_1} \cong \tilde{C_2}$, so what remains is to show that $\tau$ fixes all nodes so that this isomorphism respects the marked points.

Suppose, toward a contradiction, that $Q_i \neq Q_j$ are nodes of $C$ such that $\tau(Q_i) = Q_j$.  Note that the preimages of these nodes in $S$ must be nodes and not contracted rational components, since a contracted rational component must be sent to itself by the covering involution on $S$. 
Let $S_1$ and $S_2$ be, respectively, the preimages of $C_1$ and $C_2$ in $S$.
Denoting the preimages in $S$ of $Q_i, Q_j$ by $\hat{Q_i}, \hat{Q_j}$, the fact that $\tau(Q_i) = Q_j$ implies that $f(\hat{Q}_i) = f(\hat{Q}_j) = N$ for some node $N$ of $T$, which is necessarily a non-separating node. However, the fact that $\tau(C_1)=C_2$ means that the covering involution on $S$ swaps the corresponding components $S_1$ and $S_2$, and $f$ maps $S_1$ and $S_2$ birationally to the component $T_1$ of $T$ containing $N$.  Since $N$ is a non-separating node, $T_1$ has geometric genus at most $h-1$, contradicting the fact that it is birational to the curves $S_1$ and $S_2$ of geometric genus $h$.
\end{proof}

In light of Lemmas~\ref{lem: zeta surjective implies pullback is multiple of diagonal} and \ref{lem: tau switching components implies zeta surjective}, to finish the proof of Proposition \ref{prop:g=hlemma6}, it suffices to show that $\tau$ cannot map each component to itself. Here, the argument diverges from \cite{vZ}, and the numerical conditions in Proposition \ref{prop:g=hlemma6} are used.

\begin{lemma}\label{lem: tau cannot fix components}
    Let $g\geq 2h$.  If $m \geq 1$ or $g \geq 4h+2$, then $\tau$ cannot fix $C_1$ and $C_2$.
\end{lemma}
\begin{proof}
First, note that, due to the ordering of marked points in the definition of $j$ in \eqref{eq:defj}, the map $\tau$ must send marked points on $C_1$ to marked points on $C_2$.  Thus, $C_1$ and $C_2$ cannot be fixed by $\tau$ if $m \geq 1$.

Now, assume that $m=0$, so $g \geq 4h + 2$ by assumption. Suppose, toward a contradiction, that $\tau$ fixes both $C_1$ and $C_2$.  As in Lemma~\ref{lem: tau switching components implies zeta surjective}, we first note that the preimage in $S$ of each node $Q_i \in C$ must be a node $\hat{Q}_i \in S$ and not a contracted rational component.  To see this, note that a contracted rational component must have two marked ramification points that are not nodes (by the definition of admissible covers), so the nodes are not ramification points and are therefore swapped by the covering involution; this implies that the involution $\tau$ swaps the branches of $Q_i$, contradicting that $\tau$ fixes $C_1$ and $C_2$. 

In light of the above, the set of nodes of $S$ is $\{\hat{Q}_1, \ldots, \hat{Q}_k\}$, and each of these is sent by $f$ to a node by the definition of admissible covers.  Since $\tau$ induces an involution on the nodes of $S$, this implies that there are at least $\frac{k}{2}$ nodes in $T$.  Furthermore, each node of $T$ is non-separating; to see this, note that the complement of any node in $T$ is the image under $f$ of the complement of either one or two nodes in $S$.  Since the $k$ nodes of $S$ are all non-separating, and $k >2$ by the condition $g \geq 4h+2$, it follows that these complements are both connected.  We have therefore shown that $T$ has at least $\frac{k}{2}$ non-separating nodes.  Furthermore, $T$ has two irreducible components; this follows from the fact that $S$ has two irreducible components (since $C$ has two irreducible components and we have already seen that no component of $S$ is contracted in $C$) and both of these components are fixed by the covering involution. Therefore, the arithmetic genus of $T$ is at least
\[\frac{k}{2} -1 = \frac{g}{2} - h - \frac{1}{2}>h,\]
contradicting the fact that the arithmetic genus of $T$ is $h$.
\end{proof}

We can now give a quick proof of Proposition \ref{prop:g=hlemma6}.

\begin{proof}[Proof of Proposition \ref{prop:g=hlemma6}]
By Lemma \ref{lem: tau cannot fix components}, we must have $\tau(C_1)=C_2$, and therefore the map $\zeta:\M_{h,k}\rightarrow F$ is surjective on closed points by Lemma \ref{lem: tau switching components implies zeta surjective}. The result now follows from Lemma \ref{lem: zeta surjective implies pullback is multiple of diagonal}.
\end{proof}

\section{Holomorphic forms and pure weight cohomology}
\label{sec:weightfiltration}

In this section, we review the relevant background on holomorphic forms, and we prove, in certain genus-one and genus-two cases, existence results for holomorphic forms on $\Mbar_{g,n}$.  The key reason that such forms are of interest for us is the following.

\begin{fact}\label{fact:2}
Any nonzero element of $H^{k,0}(\Mbar_{g,n})$ or $ H^{0,k}(\Mbar_{g,n})$ has nonzero image under the restriction map $H^k(\Mbar_{g,n}) \rightarrow H^k(\M_{g,n})$. 
\end{fact}

To see this, one uses the long exact sequence in cohomology, together with the Thom isomorphism, to yield an exact sequence
\begin{equation}\label{eq:LES}
H^{k-2}(\widetilde{\d \M_{g,n}})\rightarrow H^k(\Mbar_{g,n})\rightarrow H^k(\M_{g,n}),
\end{equation}
where $\widetilde{\d \M_{g,n}}$ is the normalization of the boundary, which is a disjoint union of smooth and proper Deligne--Mumford stacks.  The first map in this exact sequence is a direct sum of morphisms of pure Hodge structures of type $(1,1)$, so after complexification, its image intersects $H^{k,0}(\Mbar_{g,n})$ and $H^{0,k}(\Mbar_{g,n})$ trivially.  Therefore, exactness implies that any nonzero element of $H^{k,0}(\Mbar_{g,n})$ or $ H^{0,k}(\Mbar_{g,n})$
maps to a nonzero element of $H^k(\M_{g,n})\otimes \CC$, which proves Fact~\ref{fact:2}.

Recalling the statement of Theorem~\ref{thm:interior} from the introduction, we will be particularly interested in the existence of holomorphic $(3h-3+k)$-forms on $\Mbar_{h,k}$.  When $h=1$, these are already known to exist when $k$ is odd and sufficiently large.

 \begin{proposition}
 \label{prop:g=1weights}
 Let $k\geq 11$ be odd and $k\neq 13$. Then $H^{k,0}(\Mbar_{1,k})\neq 0$.
 \end{proposition}
 \begin{proof}
     This is well-known; see \cite{Getzler} or \cite[Proposition 2.2]{CLPste}.
 \end{proof}

We now show that a similar result holds when $h=2$.

\begin{proposition}
 \label{prop:g=2weights}
     Let $k\geq 14$ be even. Then $H^{k+3,0}(\Mbar_{2,k})\neq 0$.
 \end{proposition}

In order to prove Proposition~\ref{prop:g=2weights}, we will make use of certain facts regarding the weight filtration on cohomology, which we now review.  For any variety or Deligne--Mumford stack $X$, Deligne's mixed Hodge theory \cite{Deligne71, Deligne74} implies that $H^k(X)=H^k(X,\QQ)$ admits an increasing filtration
\[0 = W_{-1}H^k(X) \subseteq W_0 H^k(X) \subseteq \cdots \subseteq W_{2k}H^k(X) = H^k(X),\]
called the  weight filtration.  When $X$ is smooth, the weights are bounded in the interval $[k,2k]$ \cite[Proposition 4.20]{PetersSteenbrink}; that is, we have
\[0 = W_{k-1}H^k(X) \subseteq W_k H^k(X) \subseteq \cdots \subseteq W_{2k}H^k(X) = H^k(X).\]
The lowest-weight component $W_kH^k(X) \subseteq H^k(X)$ is called the pure weight cohomology of $X$; it is a pure Hodge structure containing the classes of all algebraic cycles in $H^k(X)$.  When $X$ is both smooth and proper, we have $W_kH^k(X) = H^k(X)$.

More generally, if $\mathbb{V}$ is a polarizable variation of Hodge structures of weight $j$ on a smooth variety $X$, then $H^k(X,\mathbb{V})$ admits a weight filtration with weights bounded in the interval $[k+j,2k+j]$.  Furthermore, compactly-supported cohomology $H_c^k(X,\mathbb{V})$ also admits a weight filtration, whose weights are bounded in the interval $[j,k+j]$. 

Now, consider the case $X = \M_{g,n}$.  This is a smooth Deligne--Mumford stack and therefore $H^k(\M_{g,n})$ has weights in the interval $[k,2k]$.  In this case, the pure weight cohomology of $\M_{g,n}$ is related to the cohomology of the compactification $\Mbar_{g,n}$ as follows.

\begin{fact}  \label{fact:1}
The image of the restriction $H^k(\Mbar_{g,n}) \rightarrow H^k(\M_{g,n})$ is $W_kH^k(\M_{g,n})$.
\end{fact}

Indeed, for any Zariski open subset $U$ of a smooth proper Deligne--Mumford stack $X$ the restriction map $H^k(X) \rightarrow W_kH^k(U)$ is surjective \cite[Proposition 4.20]{PetersSteenbrink}. Because $\Mbar_{g,n}$ is smooth and proper, we have $W_kH^k(\Mbar_{g,n}) = H^k(\Mbar_{g,n})$, yielding Fact \ref{fact:1}.

Equipped with this background, we are prepared to prove Proposition~\ref{prop:g=2weights}.

\begin{proof}[Proof of Proposition \ref{prop:g=2weights}]
Let $\M_{2,k}^{\text{rt}}$ be the moduli space of genus-two curves with rational tails.  Explicitly, $\M_{2,k}^{\text{rt}}$ is the fiber product of the inclusion $\M_2\hookrightarrow \Mbar_{2}$ with the forgetful map $\Mbar_{2,k}\rightarrow \Mbar_2$, and thus it comes equipped with a smooth and proper forgetful map $f:\M_{2,k}^{\rt}\rightarrow \M_2$.  In particular, the Leray spectral sequence \[E_2^{p,q}=H^p(\M_2,R^qf_*\mathbb{Q})\Rightarrow H^{p+q}(\M_{2,k}^{\rt},\mathbb{Q})\]
    degenerates at $E_2$.
    
Because $\M_2$ is affine, $H^p(\M_2,R^qf_*\mathbb{Q})=0$ for $p> 3=\dim(\M_2)$. Moreover, when $q$ is odd, $H^p(\M_2,R^qf_*\mathbb{Q})=0$ because of the action of the hyperelliptic involution on the fibers of $R^qf_*\mathbb{Q}$. From the spectral sequence, we obtain an injection of mixed Hodge structures
    \begin{equation}\label{leray}
        0\rightarrow H^3(\M_2,R^{k}f_*\mathbb{Q})\rightarrow H^{k+3}(\M_{2,k}^{\rt}).
    \end{equation}
    In light of \eqref{leray}, to prove the proposition, it suffices to show that the $(k+3,0)$ part of $W_{k+3}H^3(\M_2, R^kf_*\mathbb{Q})$ is nonzero. Indeed, by the analogue of Fact~\ref{fact:1} for the rational-tails moduli space, the image of $H^{k+3}(\Mbar_{2,k}) \rightarrow H^{k+3}(\M^{\text{rt}}_{2,k})$ is precisely $W_{k+3}H^{k+3}(\M^{\text{rt}}_{2,k})$.  This restriction map is a morphism of pure Hodge structures, so the fact that the $(k+3,0)$ part of the Hodge structure $W_{k+3}H^{k+3}(\M^{\text{rt}}_{2,k})$ is nonzero implies that $H^{k+3,0}(\Mbar_{2,k}) \neq 0$.  
    
    By \cite[Propositon 4.5]{PetersenTommasi}, the local system $R^{k}f_*\mathbb{Q}$ contains a summand $\mathbb{V}_{\frac{k}{2}+3,\frac{k}{2}-3}$. Explicitly, $\mathbb{V}_{\frac{k}{2}+3,\frac{k}{2}-3}$ is the restriction to $\M_2$ of the symplectic local system on $\A_2$ (the moduli space of principally polarized abelian surfaces) corresponding to the representation of $\mathrm{Sp}_4$ of highest weight $(\frac{k}{2}+3,\frac{k}{2}-3)$. The Thom--Gysin exact sequence yields a right exact sequence of Hodge structures
    \begin{align}\label{abelianthomgysin}
     W_{k+1}H^{1}(\A_1\times \A_1,\mathbb{V}_{\frac{k}{2}+3,\frac{k}{2}-3})&\rightarrow W_{k+3}H^3(\A_2,\mathbb{V}_{\frac{k}{2}+3,\frac{k}{2}-3})\\
     \nonumber &\rightarrow W_{k+3}H^{3}(\M_2,\mathbb{V}_{\frac{k}{2}+3,\frac{k}{2}-3})\rightarrow 0,
    \end{align}
    where the first map is of Hodge type $(1,1)$.
    Let $H^3_{!}(\A_2,\mathbb{V}_{\frac{k}{2}+3,\frac{k}{2}-3})$ denote the image of compactly-supported cohomology under the natural map
    \[
    H^3_c(\A_2,\mathbb{V}_{\frac{k}{2}+3,\frac{k}{2}-3})\rightarrow H^3(\A_2,\mathbb{V}_{\frac{k}{2}+3,\frac{k}{2}-3}).
    \]
    Because the mixed Hodge structure $H_c^3(\A_2,\mathbb{V}_{\frac{k}{2}+3,\frac{k}{2}-3})$ has weights at most $k+3$ and $H^3(\A_2,\mathbb{V}_{\frac{k}{2}+3,\frac{k}{2}-3})$ has weights at least $k+3$, we have that $H^3_{!}(\A_2,\mathbb{V}_{\frac{k}{2}+3,\frac{k}{2}-3})$ is a pure Hodge structure of weight $k+3$. By \cite{FaltingsChai}, the $(k+3,0)$ part of the pure Hodge structure $H^3_{!}(\A_2,\mathbb{V}_{\frac{k}{2}+3,\frac{k}{2}-3})$ is naturally identified with $S_{6,\frac{k}{2}}$,
    the space of cusp forms for $\mathrm{Sp}_4$ of type $\mathrm{Sym}^6\otimes \det^{\otimes\frac{k}{2}}$; see also \cite[Section 3]{FPTautological}. This space is nonzero for even $k\geq 16$, $k\neq 18$ \cite{Tsushima,modularforms}. Moreover, for degree reasons, the $(k+3,0)$ part of 
    $H^3_{!}(\A_2,\mathbb{V}_{\frac{k}{2}+3,\frac{k}{2}-3})$ is not in the image of the first map in \eqref{abelianthomgysin}.   Hence, we obtain a sub-Hodge structure of $W_{k+3}H^{3}(\M_2,\mathbb{V}_{\frac{k}{2}+3,\frac{k}{2}-3})$ with $(k+3,0)$ part nonzero.  This proves the claim for all $k \neq 14, 18$.

    A similar argument shows that the $(k+3,0)$ part of the Hodge structure $W_{k+3}H^{k+3}(\M_{2,k}^{\rt})$ is nonzero for $k=14,18$, but this time using the local systems $\mathbb{V}_{7,7}$ and $\mathbb{V}_{9,9}$. By \cite[Theorem 2.1 and Remark 2.3]{Petersen}, $H^3(\A_2,\mathbb{V}_{7,7})$ contains a copy of the space $S_{18}$ and $H^3(\A_2,\mathbb{V}_{9,9})$ contains a copy of the space $S_{22}$, where $S_\ell$ is the pure Hodge structure corresponding to cusp forms for $\mathrm{SL}_2$ of weight $\ell$. The $(\ell-1,0)$ parts of these pure Hodge structures are nonzero when $\ell=18, 22$, so the claim is proved in these remaining cases.     
\end{proof}

\begin{remark}
    An alternative way to handle the $k=14,18$ cases of Proposition~\ref{prop:g=2weights} is to construct a holomorphic $(k+3)$-form on $\Mbar_{2,k}$ explicitly.  For $k=14$, this was carried out in \cite[Section 3.5]{FPTautological} by using the Siegel cusp form $\chi_{10}$, and working similarly with $\chi_{12}$ allows one to handle the case $k=18$.
\end{remark}

\section{Non-tautological double cover cycles}
\label{sec:mainresults}

In this section, we apply Proposition~\ref{prop:g=hlemma6} to prove Theorem~\ref{thm:interior}: namely, that $[\B_{g \rightarrow h, 0, 2m}]$ is non-tautological under a certain assumption on the existence of holomorphic forms on a moduli space of stable curves. By combining this result with the genus-one and genus-two cases in which such forms were shown to exist (Propositions~\ref{prop:g=1weights} and \ref{prop:g=2weights}), we deduce Theorem~\ref{thm:main}.

\subsection{Non-tautological double cover cycles on the interior}

First, let us recall the statement of Theorem~\ref{thm:interior}.

\begin{customthm}{B}
Let $g\geq 2h$ and set $k := g-2h+m+1$.  Suppose that
\[H^{3h-3+k,0}(\Mbar_{h,k})\neq 0.\]
If $m \geq 1$ or $g \geq 4h+2$, then
     \[[\B_{g\rightarrow h,0,2m}]\in H^{6h-6+2k}(\M_{g,2m})
    \]
    is non-tautological.
\end{customthm}

\begin{proof}
Suppose, toward a contradiction, that $[\B_{g\rightarrow h,0,2m}]$ is tautological.  Given that tautological classes on $\M_{g,2m}$ are, by definition, the restrictions of tautological classes on $\Mbar_{g,2m}$, it follows from excision that
\begin{equation}
    \label{eq:excision}
    [\Bbar_{g \rightarrow h, 0,2m}] = T+ B
\end{equation}
where $T$ is a tautological class and $B$ is an algebraic cycle pushed forward from the boundary $\d\Mbar_{g,2m}$. Pulling back both sides of \eqref{eq:excision} under $i: \Mbar_{h,k} \times \Mbar_{h,k} \rightarrow \Mbar_{g,2m}$ and applying the Hodge--K\"unneth decomposition, we obtain an element of
\[H^{3h-3+k,3h-3+k}\left(\Mbar_{h,k} \times \Mbar_{h,k}\right) \cong \bigoplus_{a+c=b+d=3h-3+k}H^{a,b}(\Mbar_{h,k}) \otimes H^{c,d}(\Mbar_{h,k}).\]
We claim that
\begin{enumerate}[label=(\alph*)]
    \item $i^*[\Bbar_{g \rightarrow h, 0,2m}]$ has a contribution in $H^{3h-3+k,0}(\Mbar_{h,k}) \otimes H^{0,3h-3+k}(\Mbar_{h,k})$;
        \item $i^*T$ has no contribution in $H^{3h-3+k,0}(\Mbar_{h,k}) \otimes H^{0,3h-3+k}(\Mbar_{h,k})$;
    \item $i^*B$ has no contribution in $H^{3h-3+k,0}(\Mbar_{h,k}) \otimes H^{0,3h-3+k}(\Mbar_{h,k})$.
\end{enumerate}
These three statements together contradict the equation
\[i^*[\Bbar_{g \rightarrow h, 0,2m}] = i^*T + i^*B,\]
so proving them will complete the proof.

To prove (a), note that by Proposition \ref{prop:g=hlemma6}, we have $i^*[\Bbar_{g \rightarrow h, 0, 2m}] = \alpha [\Delta] + B'$ for some $\alpha \in \QQ_{> 0}$ and some algebraic cycle $B'$ supported on $\partial(\M_{h,k} \times \M_{h,k})$.  The assumption that $H^{3h-3+k,0}(\Mbar_{h,k}) \neq 0$ (hence $H^{0,3h-3+k}(\Mbar_{h,k}) \neq 0$, as well) implies that there is a nonzero contribution to the Hodge--K\"unneth decomposition of $[\Delta]$ in the component $H^{3h-3+k,0}(\Mbar_{h,k}) \otimes H^{0,3h-3+k}(\Mbar_{h,k})$. The class $B'$ restricts trivially to the interior, and so it has no contribution in this component because any nonzero class in $H^{3h-3+k,0}(\Mbar_{h,k}) \otimes H^{0,3h-3+k}(\Mbar_{h,k})$ restricts nontrivially to the interior by Fact \ref{fact:2}. Therefore,
$\alpha [\Delta ] + B'$ has a nonzero contribution in $H^{3h-3+k,0}(\Mbar_{h,k}) \otimes H^{0,3h-3+k}(\Mbar_{h,k})$.

To prove (b), note that because $T$ is tautological, every component of the K\"unneth decomposition of $i^*T$ is tautological and therefore algebraic
 by \cite[Proposition 1]{GP}. Thus, the only contributions to $i^*T$ under the Hodge--K\"unneth decomposition can be in components of the form $H^{p,p}(\Mbar_{h,k}) \otimes H^{q,q}(\Mbar_{h,k})$.

What remains, then, is to prove (c).  To do so, we first note that $B$ can be assumed algebraic: indeed, both $T$ and $[\Bbar_{g \rightarrow h, 0,2m}]$ are algebraic, and the excision exact sequence can be applied in the Chow ring to obtain a representative of $B$ in Chow. Since $B$ is pushed forward from $\d\Mbar_{g,2m}$, it is a sum of classes of irreducible subvarieties, each supported on a boundary divisor of $\Mbar_{g,2m}$.  By addressing each of these classes separately, it suffices to prove (c) in the case that $B$ is supported on a single boundary divisor $D$.  This boundary divisor is the image of a map
\[i_2: \Mbar' \rightarrow \Mbar_{g,2m},\]
where $\Mbar'$ is either $\Mbar_{g-1,2m+1}$ or a product $\Mbar_{g_1,n_1+1} \times \Mbar_{g_2,n_2+1}$.  If $B$ is supported on $D$, then $B = i_{2*}(B'')$ for a class $B'' \in H^*(\Mbar')$.  From here, we consider two cases, depending on whether $D$ contains the image of $i$ or not.

First, suppose that $D$ contains the image of $i$.  (This happens, for instance if $D = D_{\text{irr}}$ is the closure of the locus of curves with a non-separating node and $g > 2h$ so that $i$ imposes at least two nodes.)  Then, similarly to the proof of \cite[Theorem 2]{vZ}, we can factor $i$ as
\[\Mbar_{h,k} \times \Mbar_{h,k} \xrightarrow{i_1} \Mbar' \xrightarrow{i_2} \Mbar_{g,2m}.\]
Thus, 
 \[i^*(B)=i_1^*i_2^*(B)=i_1^*i_2^*i_{2*}(B'')=i_1^*(c_1(N))\cdot i_1^* B''\]
by the projection formula, where $N$ is the normal bundle to the gluing map $i_2$.  Since $i_1^*(c_1(N)) \in H^2(\Mbar_{h,k} \times \Mbar_{h,k})$ is algebraic, its contribution to the Hodge--K\"unneth decomposition can only be in components of the form $H^{p,p}(\Mbar_{h,k}) \otimes H^{q,q}(\Mbar_{h,k})$ in which $2p + 2q = 2$ and hence at least one of $p$ or $q$ is nonzero.  Because the decomposition respects products, it is therefore impossible for $i_1^*(c_1(N))\cdot i_1^* B''$ to have a contribution in the component $H^{3h-3+k,0}(\Mbar_{h,k}) \otimes H^{0,3h-3+k}(\Mbar_{h,k})$.

Finally, suppose that $D$ does not contain the image of $i$.  (This happens, for instance, if $g> 2h$ and $D \neq D_{\text{irr}}$.)  Then we claim that $i^*B$ is pushed forward from $\d(\M_{h,k}\times \M_{h,k})$.  To prove this, suppose instead that $i^*B$ has nonempty intersection with $\M_{h,k} \times \M_{h,k}$. The intersection describes pairs of smooth curves $(C_1, C_2)$ whose images under $i$ land in $B$, which is contained in $D$ by assumption. For each pair of smooth curves $(C_1,C_2)$, the resulting curve $C\colonequals i(C_1,C_2)$ only has nodes arising from the pairwise gluing of the $k-m$ points, so the boundary divisors containing $C$ are fully determined by the value of $k-m$. Since $i(\M_{h,k} \times \M_{h,k})$ intersects $D$ nontrivially, it follows that $i(\M_{h,k} \times \M_{h,k})$ must lie in $D$, and by continuity, the entire image of $i$ must also lie in $D$. This is a contradiction, and so $i^*B$ is pushed forward from $\partial(\M_{h,k} \times \M_{h,k})$.

Thus, $i^*B$ maps to zero under the restriction to $H^{6h-6+2k}(\M_{h,k}\times \M_{h,k})$. Any nonzero class in $H^{3h-3+k,0}(\Mbar_{h,k})\otimes H^{0,3h-3+k}(\Mbar_{h,k})$ restricts nontrivially to the interior by Fact~\ref{fact:2}, so $i^*B$ can have no nonzero contribution in this Hodge--K\"unneth component.  This proves (c) and therefore finishes the proof.
\end{proof}

We apply Theorem~\ref{thm:interior} in the case where $h=1$ or $h=2$, in which case the values of $k$ for which the condition $H^{3h-3+k,0}(\Mbar_{h,k}) \neq 0$ of Theorem~\ref{thm:interior} is satisfied are described by Propositions \ref{prop:g=1weights} and \ref{prop:g=2weights}.

\begin{corollary}
    \label{cor:HneqRHopen}
    If $g \geq 2$ and $m \geq 0$ are such that $g + m$ is even and either $g+m=12$ or $g+m \geq 16$, then the class
    \[[\B_{g\rightarrow 1,0,2m}] \in H^{2g+2m-2}(\M_{g,2m})\] 
    is non-tautological.  If $g \geq 4$ and $m \geq 0$ are such that $g+m$ is odd and $g + m \geq 17$, then the class
    \[[\B_{g\rightarrow 2,0,2m}] \in H^{2g+2m}(\M_{g,2m})\]
    is non-tautological.
\end{corollary}

In particular, Corollary~\ref{cor:HneqRHopen} exhibits non-tautological classes on $\M_{g,2m}$ whenever $g \geq 4$ and $m \geq 0$ satisfy either $g+m=12$ or $g+m\geq 16$, and similarly for $g=2$ and $g=3$ as long as $g+m$ is even.  Thus, the proof of Theorem~\ref{thm:main} is immediate.

This line of reasoning also yields new examples of non-tautological classes on the compact moduli space: the locus of double covers of genus-two curves, extending results of Lian \cite{Lian}.

\begin{corollary}
\label{cor:g=2}
Let $g \geq 4$, and let $m \geq 0$ be such that $g+m$ is odd and $g+m \geq 17$.
Then the class
\[[\Bbar_{g\rightarrow 2,n,2m}] \in H^{2(g+m)}(\Mbar_{g,n+2m})\]
is non-tautological for all $ 0 \leq n \leq 2g-6$.
\end{corollary}
\begin{proof}
The $n=0$ case follows from Corollary~\ref{cor:HneqRHopen}. From here, one can use exactly the same argument as in \cite[Lemma 8]{vZ} to show that if $[\Bbar_{g \rightarrow 2, n, 2m}]$ is non-tautological for some $n$, then $[\Bbar_{g \rightarrow 2, n+1, 2m}]$ is also non-tautological; the key point is that the pushforward of $[\Bbar_{g\rightarrow h, n+1,2m}]$ under the forgetful morphism $\Mbar_{g,n+1+2m} \rightarrow \Mbar_{g,n+2m}$ is a scalar multiple of $[\Bbar_{g \rightarrow h,n, 2m}]$ by dimension considerations.  Thus, the corollary follows for $n \leq 2g-6$, which, by Remark~\ref{rem:r}, is the maximal number of ramification points.
\end{proof}

\subsection{Non-tautological classes on curves with rational tails}

We conclude the paper by remarking that, although Theorem~\ref{thm:main} only furnishes non-tautological classes on $\M_{g,n}$ when $n$ is even, it is possible to use it to produce non-tautological classes with odd numbers of marked points by expanding to the larger moduli space $\M^{\text{rt}}_{g,n}$ of curves with rational tails.  To do so, we denote by $\B^{\text{rt}}_{g \rightarrow h, n, 2m}$ the restriction of $\Bbar_{g \rightarrow h, n, 2m}$ to $\M^{\text{rt}}_{g,n+2m}$, and we recall that the tautological ring of $\M^{\text{rt}}_{g,n+2m}$ is defined as the image of $RH^*(\Mbar_{g,n+2m})$ under restriction.  Then we have the following result.

\begin{proposition}
\label{prop:rationaltails}
Let $1\leq n\leq 2g-4h+2$.
If $[\B^{\mathrm{rt}}_{g \rightarrow h, n-1, 2m}]$ is non-tautological, then $[\B^{\mathrm{rt}}_{g\rightarrow h, n, 2m}]$ is non-tautological.
\end{proposition}
\begin{proof}
    The forgetful map $\pi:\M_{g,n+2m}^{\rt}\rightarrow \M_{g,n-1+2m}^{\rt}$ is proper and \[
    \pi_*[\B^{\mathrm{rt}}_{g\rightarrow h, n, 2m}]=[\B^{\mathrm{rt}}_{g\rightarrow h, n-1, 2m}].
    \]
    The result follows because tautological classes push forward to tautological classes under forgetful maps.
\end{proof}

Note that if $\B_{g\rightarrow h, n, 2m}$ is non-tautological, then so is $\B^{\mathrm{rt}}_{g\rightarrow h, n, 2m}$. Thus, combining Proposition~\ref{prop:rationaltails} with Corollary~\ref{cor:HneqRHopen}, we find that there exist non-tautological algebraic cycles in $H^*(\M^{\text{rt}}_{g,n})$ whenever $g \geq 2$ and either $2g+n \in \{24,25\}$ or $2g+n \geq 32$.  Because all algebraic classes on $\Mbar_{0,n}$ and $\Mbar_{1,n}$ are already known to be tautological, this settles the question of the existence of non-tautological algebraic classes in $H^*(\M^{\text{rt}}_{g,n})$ for all but finitely many values of $g$ and $n$.
 
\bibliographystyle{alpha}
\bibliography{bibliography.bib}
	
\end{document}